\documentclass[12pt,leqno,a4paper]{amsart}
\usepackage{amssymb,enumerate}
\usepackage{hyperref}
\newcommand{\CC}{{\mathbb{C}}}
\newcommand{\GG}{{\mathbb{G}}}
\newcommand{\LL}{{\mathbb{L}}}
\newcommand{\QQ}{{\mathbb{Q}}}
\newcommand{\ZZ}{{\mathbb{Z}}}

\newcommand{\bi}{{\mathbf{i}}}
\newcommand{\bu}{{\mathbf{u}}}
\newcommand{\bx}{{\mathbf{x}}}

\newcommand{\cF}{{\mathcal{F}}}
\newcommand{\cH}{{\mathcal{H}}}

\newcommand{\Deg}{{\operatorname{Deg}}}
\newcommand{\Fam}{{\operatorname{Fam}}}
\newcommand{\Id}{{\operatorname{Id}}}
\newcommand{\Irr}{{\operatorname{Irr}}}
\newcommand{\norm}{\operatorname{N}}
\newcommand{\St}{\operatorname{St}}
\newcommand{\tr}{\operatorname{tr}}
\newcommand{\Uch}{{\operatorname{Uch}}}

\newcommand{\GL}{{\operatorname{GL}}}

\newcommand{\tw}[1]{{}^{#1}\!}

\newcommand{\Pha}[1]{{\Phi_{#1}'}}
\newcommand{\Phb}[1]{{\Phi_{#1}''}}
\newcommand{\hlf}{{\frac{1}{2}}}

\newcommand{\iso}{{\ \buildrel{1-1}\over\longrightarrow\ }}

\let\eps=\epsilon
\let\ze=\zeta
\let\la=\lambda

\let\vhi=\varphi
\let\cE=\Uch

\newtheorem{thm}{Theorem}[section]

\newtheorem{cor}[thm]{Corollary}

\newtheorem{thmA}{Theorem}

\theoremstyle{definition}
\newtheorem{rem}[thm]{Remark}
\newtheorem{exmp}[thm]{Example}

\makeatletter
\def\namedlabel#1#2{\begingroup #2%
    \def\@currentlabel{#2}%
    \phantomsection\label{#1}\endgroup
}

\begin{document}

\title[Harish-Chandra theories for spetses]{Harish-Chandra theories, Ennola
  $d$-ality\\
  and Rouquier blocks for spetses}

\author{Gunter Malle}
\address{FB Mathematik, RPTU Kaiserslautern--Landau, Postfach 3049,
  67653 Kaisers\-lautern, Germany.}
\email{malle@mathematik.uni-kl.de}

\begin{abstract}
It has been shown that the theory of unipotent characters of finite reductive
groups admits a generalisation to objects whose Weyl group is a spetsial
complex reflection group, called spetses. In this paper we prove several
natural properties satisfied by the unipotent characters of spetses, in
particular the validity of all Harish-Chandra theories as well as the
existence of Ennola $d$-alities for all integers $d$, Alvis--Curtis duality,
and compatibility with Rouquier blocks of relative Hecke algebras.
\end{abstract}



\thanks{The author gratefully acknowledges financial support by the
 DFG -- Project 286237555 -- TRR 195.}

\date{\today}

\maketitle


\section{Introduction}   \label{sec:intro}

The work of George Lusztig shows that the unipotent characters of finite
reductive groups are in some sense independent of the underlying field size
but only depend on the Weyl group, together with the action of the Frobenius.
In the beginning of the 1990s Michel Brou\'e and the author realised that a
notion of
unipotent characters might also exist for all finite Coxeter groups and even
for (certain) complex reflection groups in place of Weyl groups. Unipotent
characters for all finite Coxeter groups had in fact been constructed by
Lusztig in unpublished work in the 1980s (see \cite{Lu93}). By the end of
1994 the present author had constructed unipotent degrees for all spetsial
complex reflection groups. For the imprimitive ones, that is, the infinite
series $G(m,1,n)$ and $G(m,m,n)$, as well as for
their twisted versions, these were investigated in \cite{Ma95}. The lists of
unipotent degrees for the primitive spetsial reflection groups finally appeared
in \cite{BMM14}, where we formulated an axiomatic setting and showed that a
certain subsets of those axioms already uniquely determine the unipotent
characters in all cases. But the question whether those degrees satisfy all
of the axioms stated in \cite{BMM14} remained open. For the infinite series,
most of them had already been verified in \cite{Ma95}.

It is our aim in this paper to prove the validity of the most important axioms
from \cite{BMM14} also for the primitive spetsial groups, including the general
$\Phi$-Harish-Chandra and Howlett--Lehrer theories as well as the existence of
Alvis--Curtis duality and Ennola $d$-ality, and compatibility with Rouquier
blocks. Apart from their fundamental importance in block theory of
finite reductive groups, Harish-Chandra and Howlett--Lehrer theories have
recently become relevant in our work on character tables of spetses \cite{KMS}
as well as in our generalisation of the Trinh--Xue conjecture on block
intersections \cite{CM26}.

Our main results are collected in the following two theorems:

\begin{thmA}   \label{thm:main}
 The unipotent characters of simple spetses are consistent with all
 Harish-Chandra and Howlett--Lehrer theories, with Ennola $d$-ality, as well
 as with Rouquier blocks. In particular, they satisfy the axioms~4.31, 4.32,
 4.34 and 4.35 from \cite{BMM14}.
\end{thmA}

For \emph{involutive} spetses, that is, those whose underlying reflection
group only contains reflections of order two, we have moreover:

\begin{thmA}   \label{thm:main2}
 The unipotent characters of simple involutive spetses possess an
 Alvis--Curtis duality, that is, they satisfy the axioms 4.27 and 4.28 from
 \cite{BMM14}.
\end{thmA}

The statements of the listed axioms will be recalled in the proofs below.
It should be noted that for the  spetsial reflection groups in the infinite
series most of these were already known to hold by \cite{Ma95}; this will be
made precise below.
\medskip

The paper is structured as follows: In
Section~\ref{sec:HC} we discuss Harish-Chandra theories, showing axioms~4.31
and 4.32. The existence and properties of Alvis--Curtis duality (axioms~4.27
and~4.28) is studied in Section~\ref{sec:duality}. Ennola maps are considered
in Section~\ref{sec:Ennola}, where we prove axioms~4.34 and~4.35. In
Section~\ref{sec:fam} we discuss the properties of families and compatibility
with Rouquier blocks of relative Hecke algebras, thereby completing the proofs
of Theorems~\ref{thm:main} and~\ref{thm:main2}. In Section~\ref{sec:prob} we
present some open problems on spetses and their unipotent characters. Tables
containing our explicit results are collected in Section~\ref{sec:tab}.

\section{$\Phi$-Harish-Chandra theories}   \label{sec:HC}

\subsection{Spetses and their unipotent characters}
For an introduction to terminology and basic results on spetses we refer the
reader to \cite{BMM99,BMM14}. Recall that given a spetsial complex reflection
group $W\le\GL(V)$ on a finite dimensional complex vector space $V$, and an
element $\vhi\in\norm_{\GL(V)}(W)$, we call $\GG=(V,W\vhi)$ the associated
\emph{spets}. As usual we denote by $|\GG|\in K[x]$ its \emph{order polynomial},
where $K:=\QQ(\{\tr(w\vhi)\mid w\in W\})$ is the \emph{field of definition}
of~$\GG$. It was shown in \cite{Ma95,BMM14} that attached to $\GG$ there is a
set $\Uch(\GG)$ of \emph{unipotent characters}, with degree function
$\deg:\Uch(\GG)\to K[x]$, satisfying various natural properties, called
\emph{axioms} in \cite{BMM14}, generalising those of unipotent characters of
finite reductive groups with Weyl group $W$ and their degree polynomials. For
simplicity, we will say $\GG$ is \emph{primitive} if $W$ is a primitive
reflection group, and \emph{simple} if $W$ is irreducible.

The $\Phi$-Harish-Chandra theories lie at the heart of the construction of
$\Uch(\GG)$, while the degrees are given by a type of Howlett--Lehrer formulas.
In fact, the unipotent degrees for primitive spetses were derived in
\cite{BMM14} from the assumption that at least certain $\Phi$-Harish-Chandra
theories are satisfied.

For any irreducible cyclotomic polynomial $\Phi$ over the field of definition
$K$ of $\GG$, Harish-Chandra theory (should) provide a partition of the
unipotent characters $\Uch(\GG)$ into \emph{Harish-Chandra series
$\cE(\GG,(\LL,\la))\subseteq\Uch(\GG)$} above $\Phi$-cuspidal pairs $(\LL,\la)$
of $\GG$, each of them in bijection with the irreducible characters
of the relative Weyl group $W_\GG(\LL,\la)$ of $(\LL,\la)$ in $\GG$. It is
one of the mysteries of the theory that the relative Weyl groups turn out, in
a natural way, to always be complex reflection groups in their own right.
Thus, attached to $W_\GG(\LL,\la)$ is a generic Hecke algebra
$\cH:=\cH(W_\GG(\LL,\la),\bu)$ over $\ZZ_K[\bu,\bu^{-1}]$, where $\ZZ_K$ is
the ring of integers of the number field $K$ and $\bu=(u_{s,j})_{s,j}$ are
indeterminates with $s$ running over conjugacy class representatives of
\emph{distinguished reflections $s\in W_\GG(\LL,\la)$} and $1\le j\le o(s)$
(see \cite[1C]{BMM99}).
Let $\ze$ be a root of $\Phi$. A \emph{cyclotomic specialisation of $\cH$} is
then the image of $\cH$ under a specialisation induced by an assignment
$$u_{s,j}\mapsto \exp(2\pi\bi j/o(s))\,v^{a_{s,j}}
  \qquad(\text{$s$ distinguished, }1\le j\le o(s)),$$
with $a_{s,j}\in\ZZ$, where $v^{|Z(W)|}=\ze^{-1}x$. Then,
\begin{equation}   \label{eq:param}
\bx:=\big(\exp(2\pi\bi j/o(s))\,v^{a_{s,j}}\big)_{s,j}
\end{equation}
will be called an \emph{admissible set of parameters for $\cH$}. Note that any
such cyclotomic specialisation $\cH(W_\GG(\LL,\la),\bx)$ is semisimple (since
by further specialising $v\mapsto1$, so $x\mapsto\ze$, we obtain the group
algebra of $W_\GG(\LL,\la)$) and thus, by Tits's deformation theorem and the
proven freeness conjecture \cite{Ch18,Ma19,Ts20}, a deformation of the
group algebra $\CC W_\GG(\LL,\la)$. In particular, the irreducible characters
of $\cH(W_\GG(\LL,\la),\bx)$ over a splitting field are in natural bijection
with $\Irr(W_\GG(\LL,\la))$. 

Now, the degrees of the unipotent characters in $\cE(\GG,(\LL,\la))$
should be given by the generic degrees of a suitable cyclotomic Hecke algebra
specialisation $\cH(W_\GG(\LL,\la),\bx)$. This latter property is sometimes
called a \emph{$\Phi$-Howlett--Lehrer theory}.

\subsection{Harish-Chandra and Howlett--Lehrer theory}   \label{subsec:HL}
Let $\Phi$ be a cyclotomic polynomial over $K$, that is, an irreducible divisor
in $K[x]$ of $x^n-1$ for some $n\ge1$. A pair $(\LL,\la)$, with $\LL$ a
$\Phi$-split Levi subspets of $\GG$ and $\la$ a $\Phi$-cuspidal unipotent
character of~$\LL$, is called a \emph{$\Phi$-cuspidal pair} for~$\GG$. Here,
by \cite[Def.~4.5]{BMM14}, $\la\in\Uch(\LL)$ is \emph{$\Phi$-cuspidal} if
$$\Deg(\la)_\Phi = \dfrac{|\LL|_\Phi}{|Z\LL|_\Phi},$$
where $Z\LL$ is the centre of $\LL$. This agrees with, and thus extends, the
notion of $d$-cuspidality if $W$ is a Weyl group and $\Phi$ is the $d$th
cyclotomic polynomial, see \cite[Prop.~2.9]{BMM93}.

Our first main result is what is called Axiom 4.31 in \cite{BMM14}, shown there
only for particular choices of $\Phi$ for any given primitive $W$:

\begin{thm}[$\Phi$-Harish--Chandra and $\Phi$-Howlett--Lehrer theory]   \label{thm:HC+HL}
 Let $\GG=(V,W\vhi)$ be a simple spets and $\Phi$ a cyclotomic
 polynomial over $K$. There is a partition
 $$\Uch(\GG) = \bigsqcup_{(\LL,\la)} \cE(\GG,(\LL,\la))$$
 where $(\LL,\la)$ runs over a complete set of representatives of the
 $W$-orbits on $\Phi$-cuspidal pairs of $\GG$. Moreover, for each
 $\Phi$-cuspidal pair $(\LL,\la)$ we have:
 \begin{enumerate}[\rm(a)]
  \item $W_\GG(\LL,\lambda)$ is a reflection group on the orthogonal of the
   intersection of the hyperplanes of $W_\LL$, the Weyl group of $\LL$.
  \item There exists a cyclotomic specialisation
   $\cH_\GG(\LL,\la)=\cH(W_\GG(\LL,\la),\bx)$ of the generic Hecke algebra for
   $W_\GG(\LL,\la)$, a bijection
   $$I_{\LL,\la}^\GG:\Irr(\cH_\GG(\LL,\la))\iso\cE(\GG,(\LL,\la)),
       \quad \chi \mapsto \rho_\chi \,,$$
   and a collection $\{\eps_\chi\mid\chi\in\Irr(\cH_\GG(\LL,\la))\}$ of signs,
   with the following properties:
   \begin{enumerate}
    \item[\rm(1)] The bijection $I_{\LL,\la}^\GG$ is
     $N_{\GL(V)}(W\vhi)/W$-equivariant.
    \item[\rm(2)] The Schur elements $S_\chi$ of $\cH_\GG(\LL,\la)$ satisfy
     $$\Deg({\rho_\chi})=\eps_\chi\Deg(\la)\frac{|\GG|_{x'}}{|\LL|_{x'}}
        S_\chi^{-1}\qquad\text{for all $\chi\in\Irr(\cH_\GG(\LL,\la))$}\,.$$
   \end{enumerate}
  \end{enumerate}
\end{thm}

Here, the Schur elements $S_\chi$ were explicitly determined in
\cite{Ma97,Ma98,GIM} with respect to a canonical symmetrising form on
$\cH_\GG(\LL,\la)$ predicted in \cite[Thm--Ass.~2.1]{BMM99} and shown to exist
at present in all but finitely many cases.

\begin{rem}
 Assume $W$ is a Weyl group. If $\Phi=x-1$, the above is of course the
 classical result of Howlett and Lehrer (see e.g. \cite[Thm~3.2.7]{GM20}).
 For arbitrary cyclotomic polynomials $\Phi$ over $\QQ$ the first part is one
 of the main results of \cite{BMM93}, while~(2) was shown in
 \cite[Prop.~7.1]{Ma00}.
\end{rem}

\begin{proof}[Proof of Theorem~\ref{thm:HC+HL}]
For the infinite series of irreducible reflection groups and non-exceptional
twistings this was shown in \cite{Ma95}. This leaves the finitely many
primitive groups and the exceptional twisting of order~4 of the imprimitive
group $G(3,3,3)$ (see \cite[Prop.~3.13]{BMM99}). By \cite{Ma98} the spetsial
primitive non-real reflection groups are
$$G_4, G_6, G_8, G_{14}, G_{23}=H_3, G_{24}, G_{25}, G_{26}, G_{27}, G_{29},
  G_{30}=H_4, G_{32}, G_{33}, \text{ and } G_{34}$$
in the notation of Shephard and Todd, plus the Weyl groups of types $F_4, E_6,
E_7$ and $E_8$. For the latter, the claim was already shown in
\cite[Prop.~7.1]{Ma00}. This leaves just a finite number of groups $W$ to
consider, each with a finite number of possible non-trivial twistings~$\vhi$,
and for each of them just a finite number of relevant cyclotomic polynomials
$\Phi$ over the field of definition $K$ (namely those dividing the order
polynomial of $\GG=(V,W\vhi))$. Thus, our assertion can be checked by computer
once suitable parameters of the occurring relative Hecke algebras
$\cH_\GG(\LL,\la)$ can somehow be obtained. In fact, with few exceptions,
those had already been found by the author in 1994 in the course of
constructing the unipotent characters. The remaining ones were derived in the
course of writing this paper. They are collected in Table~\ref{tab1} for the
non-Coxeter groups. (The parameters for Coxeter groups were already published
in \cite[Tab.~1 and~3]{BM93}, respectively already by Lusztig for Coxeter tori
in \cite{Lu77}.) The necessary calculations for the determination
of relative Weyl groups and for the verification of the statements were partly
done using Jean Michel's GAP programs \cite{Mi15}, e.g.~{\tt CuspidalPairs},
partly in {\tt maple}.
\end{proof}

\begin{rem}
Inspection of Table~\ref{tab1} shows that, as already for the exceptional Weyl
groups, the most interesting $\Phi$-local structures arguably occur for
$\Phi(\ze_4)=0$ (where $\ze_4$ is a primitive 4th root of unity).
As in types $E_7$ and $E_8$ there exist $\Phi$-cuspidal pairs in $\GG$ whose
relative Weyl group is not well-generated: of types $G(4,2,2)$ in $E_7$,
$G_{31}$ in $E_8$, and $G_7$ in $G_{34}$. Neither $G_7$ nor $G_{31}$ occur for
any other cuspidal pairs in spetses. Also, as in $E_7$ there are
$\Phi$-cuspidal pairs $(\LL,\la)$ with $W_\GG(\LL,\la)$ a proper subgroup of
$W_\GG(\LL)$, namely $G_7<G_{10}$ in $G_{34}$; that is, there is a non-trivial
action of $W_\GG(\LL)$ on the set of $\Phi$-cuspidal characters of~$\LL$.
Furthermore, the exceptional twisted spets $^4G(3,3,3)$ which in
Table~\ref{tab1} we denote $\tw4I_3^{(3)}$ occurs as a component of
$\Phi$-split Levi subgroups for $G_{33}$ and $G_{34}$.   \par
Incidentally, this shows that even to prove all $\Phi$-Harish-Chandra theories
just for split spetses it is necessary to understand the exotic twisted spets
$\tw4I_3^{(3)}$. Since those had not been published so far, we give its
unipotent degrees in Table~\ref{tab:4I3}. The families (see
Section~\ref{sec:fam}) are distinguished by the power of $x$ dividing the
degree polynomials.
\end{rem}

\begin{table}[htb]
\caption{Degrees of unipotent characters of $\tw4I_3^{(3)}$}   \label{tab:4I3}
$\begin{array}{cccc}
 1&\quad&                  \frac{\sqrt{-3}}{3}\ze_3^2x^4\Phb3\Pha6\Phb{12}\\
 \frac{\sqrt{-3}}{3}\ze_3^2x\Phb3\Pha6\Phb{12}&& \frac{\sqrt{-3}}{3}\ze_3x^4\Pha3\Phb6\Pha{12}\\
 \frac{\sqrt{-3}}{3}\ze_3x\Pha3\Phb6\Pha{12}&& \frac{\sqrt{-3}}{3}x^4\Phi_1\Phi_2\Phi_4\\
 \frac{\sqrt{-3}}{3}x\Phi_1\Phi_2\Phi_4&& x^9
\end{array}$
\end{table}

Here $\Phi_d'$ denotes the cyclotomic polynomial over $\QQ(\ze_3)$ with
$\Phi_d'(\ze_d)=0$ (where $\ze_d:=\exp(2\pi\bi/d)$), and $\Phi_d''$ its
complex conjugate, for $d\in\{3,6,12\}$.

It was shown in \cite[Cor.~7.2]{Ma99} that $K(x^{1/|Z(W)|})$ is a splitting
field for the Hecke algebra $\cH(\GG,(W,1))$ of the principal $x-1$-series
of~$\GG$. Inspection of Table~\ref{tab1} shows the following: for any
primitive spets $\GG=(V,W\vhi)$ and any cyclotomic polynomial $\Phi$ over $K$,
the parameters of all relative Hecke algebras $\cH_\GG(\LL,\la)$ of
$\Phi$-cuspidal pairs $(\LL,\la)$ with cyclic relative Weyl group
$W_\GG(\LL,\la)$ lie in $K(x^{1/|Z(W)|},\ze)$, where $\ze$ is a suitable root
of unity. But in fact, most of the time a stronger statement holds:

\begin{cor}   \label{cor:irr}
 Let $\GG$ be a simple spets and $(\LL,\la)$ a $\Phi$-cuspidal pair of
 $\GG$ for some cyclotomic polynomial $\Phi$. If the relative Weyl group
 $W_\GG(\LL,\la)$ is not cyclic, then all parameters of the relative Hecke
 algebra $\cH_\GG(\LL,\la)$ already lie in $K(x)$.
\end{cor}

\begin{proof}
It was shown in \cite[Satz~3.14 and~6.10]{Ma95} that the parameters for
spetses related to the groups $G(m,1,n)$ and $(G,m,n,n)$ are always contained
in $K(x)$. For the exceptional Coxeter groups the claim can be checked from
\cite[Tab.~8.1 and~8.3]{BM93}, and for the non-real primitive spetses it is
immediate by inspection of Table~\ref{tab1}.
\end{proof}

\subsection{Local determination of parameters}
For finite reductive groups, the parameters of relative Hecke algebras
are known to be determined `locally', that is, already inside minimal
$\Phi$-split Levi overgroups of the centraliser of a Sylow $\Phi$-torus.
The same turns out to hold for general spetses, and in fact this was the way
the correct parameters for the $\Phi$-cuspidal pairs with non-cyclic relative
Weyl group were constructed in the first place. (For the cyclic cases they
had essentially to be guessed.)

To be more precise, let $(\LL,\la)$ be a $\Phi$-cuspidal pair of
$\GG=(V,W\vhi)$, where $\LL = (V,W_\LL w\vhi)$ for some $w\in W$.
The parameters of the associated Hecke algebra
$$\cH_\GG(\LL,\la)=\cH(W_\GG(\LL,\la),\bx)$$
are of the form $\bx=(x_{s,i})_{s,i}$ as in~(\ref{eq:param}), where $s$ runs
over a system of representatives of the $W_\GG(\LL,\la)$-classes of
distinguished reflections in $W_\GG(\LL,\la)$, and $1\le i\le o(s)$.

For any reflecting hyperplane $H\le V$ of $W_\GG(\LL,\la)$ we denote by $\GG_H$
the \emph{Levi subdatum} of $\GG$ defined by $\GG_H := (V,W_Hw\vhi)$ where
$W_H$ is the pointwise stabilizer of $H$.
Then $W_{\GG_H}(\LL,\la)$ is cyclic and contains a unique distinguished
reflection $s_H$ of $W_\GG(\LL,\la)$.

Inspection of the lists of parameters in Table~\ref{tab1} allows one to verify
the following (this is Axiom~4.32 in \cite{BMM14}):

\begin{cor}   \label{cor:param}
 In the situation of Theorem~\ref{thm:HC+HL}, for any $\Phi$-cuspidal pair
 $(\LL,\la)$ and any Levi subdatum $\GG_H$ of~$\GG$ as above the parameters of
 $\cH_{\GG_H}(\LL,\la)$ are the same as the parameters of $\cH_\GG(\LL,\la)$
 corresponding to the distinguished reflection $s_H$.
\end{cor}

\begin{exmp}
 Let $\GG=(\CC^6,G_{34})$ and $\Phi_4=x^2+1$. Let $\LL$ be the centraliser
in~$\GG$ of a Sylow $\Phi$-torus, with relative Weyl group $W_\GG(\LL)=G_{10}$,
a primitive complex reflection group generated by a reflection of order~4 and
one of order~3. Now $\GG$ has $\Phi_4$-split Levi subdata $\Phi_4.\tw2D_4$ and
$\Phi_4.\tw4I_3^{(3)}A_1$ in which the relative Weyl groups of Sylow $4$-tori
are cyclic, of respective orders~4 and~3. The parameters of their cyclic
relative Hecke algebras are indeed the ones given for the principal
$\Phi_4$-series of~$\GG$ in Table~\ref{tab1}. The parameters at the generator
of order~4 are in fact the same as in the finite reductive groups of type~$E_7$
(see \cite[Tab.~5.12]{BM93}), since these also possess a $\Phi_4$-split Levi
subgroup of type $\tw2D_4$.
\end{exmp}

\begin{rem}
 In view of Corollary~\ref{cor:irr}, Corollary~\ref{cor:param} shows that
 relative Hecke algebras with parameters not contained in $K(x)$ cannot occur
 as proper parabolic subalgebras of relative Hecke algebras in simple
 spetses of larger rank.
\end{rem}

\section{Alvis--Curtis duality}   \label{sec:duality}

For a reflection group $W$ we denote by $N_W$ its number of reflections. This
agrees with its number of reflecting hyperplanes if all reflections in $W$
have order~2. Clearly, in the latter case any parabolic subgroup of $W$ is
also generated by reflections of order~2. The following statement contains,
in particular, Axioms~4.27 and~4.28 of \cite{BMM14}, but parts~(b) and~(c)
go beyond:

\begin{thm}   \label{thm:duality}
 Let $\GG=(V,W\vhi)$ be an involutive spets. There is an involution $D_\GG$ on
 $\Uch(\GG)$ such that
 $$\Deg(D_\GG(\rho))=\pm x^{N_W}\overline{\Deg(\rho)|_{x^{-1}}}\qquad
   \text{for all $\rho\in\Uch(\GG)$}.$$
 Moreover, for any cyclotomic polynomial $\Phi$ over $K$ and any
 $\Phi$-cuspidal pair $(\LL,\la)$ of $\GG$ we have
 \begin{enumerate}
  \item[\rm(a)] $(\LL,D_\LL(\la))$ is also a $\Phi$-cuspidal pair with
   $W_\GG(\LL,D_\LL(\la))=W_\GG(\LL,\la)$;
  \item[\rm(b)] for any $\rho\in\cE(\GG,(\LL,\la))$ we have
   $D_\GG(\rho)\in\cE(\GG,(\LL,D_\LL(\la)))$;
 \item[\rm(c)] the parameters of the relative Hecke algebra
  $\cH_\GG(\LL,D_\LL(\la))$ are given by $(x_{s,j}^{-1})_{s,j}$, where
  $(x_{s,j})_{s,j}$ are the parameters of $\cH_\GG(\LL,\la)$.
 \end{enumerate}
  The bijection $D_\GG$ induces a permutation of $\Fam(\GG)$.
 \end{thm}

In part~(a), $D_\LL(\la)$ is $\Phi$-cuspidal since
$x^{\deg\Phi}\overline{\Phi|_{x^{-1}}}=\Phi$ for any cyclotomic
polynomial~$\Phi$, where $\overline{f}$ denotes the complex conjugate of
$f\in K[x]$. For~(c), recall that the parameters at any reflection are
determined only up to a common non-zero factor, as the Schur elements are
homogeneous of degree zero. The notion of families in $\Uch(\GG)$ will be
made precise in Section~\ref{sec:fam}.

\begin{proof}
The claim immediately reduces to simple spetses by \cite[Ax.~4.1]{BMM14}.
For Weyl groups $D_\GG$ is the classical Alvis--Curtis duality, see e.g.\ 
\cite[\S3.4]{GM20}. The complex imprimitive spetsial groups generated by
reflections of order~2 are the groups $G(m,m,n)$, for which the claim is
shown in \cite[6B]{Ma95}. This leaves the finitely many primitive spetsial
groups $G_{23}=H_3,G_{24},G_{27},G_{29},G_{30}=H_4,G_{33}$ and $G_{34}$. Here
the result follows by inspection of the lists of degrees in \cite{BMM14} and
the lists of parameters in Table~\ref{tab1}, respectively from
\cite[Tab.~8.3]{BM93} and \cite{Lu93} for $H_3$ and $H_4$.
\end{proof}

\section{Ennola $d$-ality}   \label{sec:Ennola}

It had been observed by Ennola that the character degree polynomials of finite
general unitary groups can be obtained from those of general linear groups
by formally replacing $q$ by $-q$, and adjusting the sign if necessary. In
fact, a similar phenomenon holds for arbitrary finite reductive groups and is
now called Ennola duality \cite[Thm~3.3]{BMM93}. In order to extend it to
spetses, let's make the following definition.

For $\GG=(V,W\vhi)$ a spets and $z$ a root of unity we define the
\emph{$z$-Ennola twist} of $\GG$ as
$$\GG_z := (V,Wz\vhi)\,.$$

The following combines Axioms~4.34 and~4.35 in \cite{BMM14}:

\begin{thm}   \label{thm:Ennola}
 Let $\GG=(V,W\vhi)$ be a simple spets.
 Let $\xi$ be a root of unity and $z:=\xi^{|Z(W)|}$. There is a bijection
 $$E_\xi\,:\, \Uch(\GG) \iso \Uch(\GG_z)$$
 with the following properties.
 \begin{enumerate}
  \item[\rm(a)] $E_\xi$ is $N_{\GL(V)}(W\vhi)/W$-equivariant; and
  \item[\rm(b)] for all $\rho\in\Uch(\GG)$, we have
      $$\Deg(E_\xi(\rho))(x)= \pm\Deg(\rho)(z^{-1} x)\,. $$
  \end{enumerate} 
 Moreover, for any cyclotomic polynomial $\Phi$ over $K$ and any
 $\Phi$-cuspidal pair $(\LL,\la)$ of $\GG$
   \begin{enumerate}
   \item[\rm(c)] $(\LL_z,E_\xi(\la))$ is a $\Phi(z^{-1} x)$-cuspidal pair of $\GG_z$;
   \item[\rm(d)] $E_\xi$ induces a bijection 
    $$\cE(\GG,(\LL,\la)) \iso \cE(\GG_z,(\LL_z,E_\xi(\la))\,;$$
   \item[\rm(e)] The parameters of the relative Hecke algebra 
    $\cH_{\GG_z}(\LL_z,E_\xi(\la))$ are obtained from those of
    $\cH_\GG(\LL,\la)$ by replacing $x^{1/|Z(W)|}$ by $\xi^{-1} x^{1/|Z(W)|}$
    throughout.
  \end{enumerate}
  The bijection $E_\xi$ sends families of $\Uch(\GG)$ to families of $\Uch(\GG_z)$.
\end{thm}

\begin{proof}
For $W$ a Weyl group this was shown in \cite[Thm~3.3]{BMM93}. For the infinite
series of spetsial groups this is a consequence of the explicit results on
parameters of relative Hecke algebras in \cite[\S3 and~\S6]{Ma95}. For the
remaining primitive spetsial groups the claim follows from the tables of
degrees in \cite{BMM14} and the lists of parameters obtained in the proof of
Theorem~\ref{thm:HC+HL}, respectively from \cite[Tab.~8.3]{BM93} and
\cite{Lu93} for $H_3$ and $H_4$.
\end{proof}

Observe that if $z\Id\in Z(W)$ then $\GG_z=\GG$ and $z\in K$. In this case,
$E_\xi$ is a permutation of $\Uch(\GG)$ of order dividing the order of $\xi$,
inducing a permutation of the families in $\Uch(\GG)$. In {\tt Chevie}
\cite{Mi15} this can be obtained by the command {\tt Ennola}.

\begin{exmp}   \label{exmp:Ennola}
 In most cases the order of $E_\xi$ divides the order of $z=\xi^{|Z(W)|}$, but
 not always. The exceptional 4-element families in finite reductive groups of
 types $E_7$ and $E_8$ are examples where $E_I$ has order~4. Similar examples
 occur e.g.~in $G_{23}=H_3$, $G_{33}$ and~$G_{34}$. Furthermore there are
 families with nine elements, e.g.~in $G_{14}, G_{27}$ and $ G_{34}$ on which
 $E_\xi$ acts regularly while $|z|$ divides~6. For the primitive spetses the
 order of $E_\xi$, for $\xi$ a primitive $|Z(W)|^2$th root of unity, on each
 family is recorded in Table~\ref{tab:fam}.
\end{exmp}

\section{Families and Rouquier blocks}   \label{sec:fam}

Let $W$ be a complex reflection group. Let $\ZZ_K$ denote the ring of integers
of its field of definition~$K$. Let $v$ be an indeterminate such that
$v^{|Z(W)|}=x$. The $\ZZ_K$-subalgebra
$$R_K(v):=\ZZ_K[v,v^{-1},(v^n-1)^{-1}_{n\geq 1}]$$
of $K(v)$ is called \emph{Rouquier ring} of $W$.

For $\bx\subset K(v)$ an admissible set of parameters for $\cH(W)$ (as
in~\ref{eq:param}), the \emph{Rouquier blocks} of the Hecke algebra
$\cH(W,\bx)$ are the blocks of the algebra
$\cH_R:=R_K(v)\cH(W,\bx)$. It has been shown by Rouquier that if $W$ is a Weyl
group and $\cH(W,\bx)$ is its usual Iwahori--Hecke algebra, then the partition
of $\Irr(\cH_R)$ into Rouquier blocks coincides (under the natural
bijection $\Irr(\cH_R)\to\Irr(W)$ from Tits's deformation theorem) with the
partition into families of $\Irr(W)$ as defined by Lusztig. In this sense,
the Rouquier blocks generalise the notion of ``families of unipotent
characters'' to cyclotomic specialisations of Hecke algebras of complex
reflection groups. Rouquier blocks for Hecke algebras of imprimitive groups
were completely determined by Chlouveraki \cite{Ch08,Ch10}, while \cite{MR03}
contains results for certain choices of parameters in the primitive case.

For a rational function $f\in K(x)$ let us write $a(f)$ for the order of zero
of~$f$ at $x=0$, and $A(f):=\deg f$. It was shown in \cite[Lem.~2.8]{MR03}
that $a(S_\chi)+A(S_\chi)$ is constant on all Schur elements $S_\chi$ for
$\chi$ running over characters in a Rouquier block of $\cH(W,\bx)$. By
\cite{Ma95} (for the infinite series), \cite[Thm~5.3]{BMM14} (for the
exceptional groups) and Table~\ref{tab:4I3} for $\tw4I_3^{(3)}$, for any
simple spets $\GG$ there is a partition
$$\Uch(\GG)=\bigsqcup_{\cF\in\Fam(\GG)}\cF$$
of the unipotent characters into families such that $a$ and $A$ are
constant on any family~$\cF$ (we can thus write $a(\cF)$ and $A(\cF)$).

This partition is compatible with the Rouquier blocks of all relative
Hecke algebras in the following sense; this is new even for the infinite series:

\begin{thm}   \label{thm:fam}
 Let $\GG=(V,W\vhi)$ be a simple spets. For each cyclotomic polynomial
 $\Phi$, for each $\Phi$-cuspidal pair $(\LL,\la)$ of $\GG$, and for each
 Rouquier block $B\subseteq\Irr\big(\cH_\GG(\LL,\la)\big)$ there is a unique
 family $\cF\in\Fam(\GG)$ such that
 $$I_{\LL,\la}^\GG(B)\subseteq \cF\cap\cE(\GG,(\LL,\la))\,,$$
 where $I_{\LL,\la}^\GG:\Irr\big(\cH_\GG(\LL,\la)\big)\iso\cE(\GG,(\LL,\la))$
 is the bijection from Theorem~\ref{thm:HC+HL}.
\end{thm}

\begin{proof}
If $W$ is a Weyl group, or of type $H_4$, the assertion was shown in
\cite[Thm~6.1]{MR03}. So we may assume $W$ is either $H_3$ or a non-real
spetsial group. For spetses with Weyl group $W=G(m,1,n)$ or $G(m,m,n)$ for
the exceptional twisting $\tw4I_3^{(3)}$ the unipotent characters are
parametrised by (equivalences classes of) symbols of degree $n$ with $m$ rows,
see \cite{Ma95}. Assume $\chi_1,\chi_2\in\Irr(\cH_\GG(\LL,\la))$ lie in the
same Rouquier block. By \cite[Thme~3.7]{BK02} this implies that $\chi_1,\chi_2$
correspond to characters of the relative Weyl group $W_\GG(\LL,\la)$ labelled
by symbols with the same multiset of entries. (Note that the proof of this
direction of the result in loc.~cit.~is indeed correct.) Then by the algorithm
in the proofs of \cite[Satz~3.14 and~6.10]{Ma95} the characters
$\rho_i:=I_{\LL,\la}^\GG(\chi_i)\in\Uch(\GG,(\LL,\la))$, $i=1,2$, are labelled
by reduced symbols with the same multiset of entries and hence lie in the same
family of $\Uch(\GG)$ by \cite[4C and~6D]{Ma95}.

Finally, assume $W$ is $H_3$ or one of the exceptional non-real spetsial
groups, or $\GG=\tw4I_3^{(3)}$. Let $(\LL,\la)$ be a $\Phi$-cuspidal pair and
$\cH:=\cH_\GG(\LL,\la)$. Again by \cite[Thme~3.7]{BK02}, if
$\chi_1,\chi_2\in\Irr(\cH)$ lie in the same Rouquier block then
$$a(S_1)+A(S_1)=a(S_2)+A(S_2)$$
for the Schur elements $S_i$ of $\chi_i$ and hence by Theorem~\ref{thm:Ennola}(b)(2) also
$$a(\rho_1)+A(\rho_1)=a(\rho_2)+A(\rho_2)$$
for the corresponding unipotent characters $\rho_i:=I_{\LL,\la}^\GG(\chi_i)
\in\Uch(\GG,(\LL,\la))$. Since the quantity $a(\rho)+A(\rho)$ is constant on
all $\rho$ in a fixed family of~$\Uch(\GG)$, our claim follows if the families
of $\GG$ are uniquely determined by this invariant. By inspection of
Table~\ref{tab:fam} the only cases among primitive spetses $\GG$ where
different families $\cF\subset\Uch(\GG)$ share the same value $a(\cF)+A(\cF)$
are as listed in Table~\ref{tab:equal} (up to Alvis--Curtis duality,
see Theorem~\ref{thm:duality}).

\begin{table}[htbp]
\caption{Families $\cF$ with equal $a(\cF)+A(\cF)$}   \label{tab:equal}
$\begin{array}{c|cl}
 \GG& \ a(\cF)+A(\cF)& |\cF|\\
\hline
 G_{26}&  30\ & 12,\ 3\\
 G_{29}&  32\ & 4,\ 1\\
 G_{33}&  42\ & 3,\ 1,\ 1\\
 G_{34}&  90\ & 9,\ 1\\
       & 108\ & 9,\ 4,\ 1\\
       & 114\ & 9,\ 3\\
       & 120\ & 36,\ 3\\
       & 126\ & 44,\ 4\\
\hline
\end{array}$
\end{table}

Now note that the families $\cF$ of $\GG$ with one element have the property
that the corresponding Schur element $S_\chi$ is a unit in $R_K(v)$ and thus
$\{\chi\}$ is a singleton block by \cite[Lem.~2.6(b)]{MR03}, as desired.
For the remaining five cases, we need to check that at most one of the two
non-singleton families does contain elements from the relevant
$\Phi$-Harish-Chandra series, which is straightforward. For example, for the
two families for $G_{26}$ all characters in one of them are of $\Phi_4$-defect
zero by the list in \cite{BMM14}, so cannot lie in
the principal $\Phi_4$-series, and all characters in both families are of
$\Phi_9$-defect zero and hence do not lie in the principal $\Phi_9'$-series.
\end{proof}

In Axiom~4.33 of \cite{BMM14} we had proposed an even stronger compatibility:
For each cyclotomic polynomial $\Phi$ and each $\Phi$-cuspidal pair $(\LL,\la)$
of $\GG$, the partition
$$\cE(\GG,(\LL,\la)) = \bigsqcup_{\cF\in\Fam(\GG)}
                                   \big(\cF\cap\cE(\GG,(\LL,\la))\big)$$
composed with the bijection
$$(I_{\LL,\la}^\GG)^{-1}:\cE(\GG,(\LL,\la))\iso\Irr\big(\cH_\GG(\LL,\la)\big)$$
should \emph{equal} the partition of $\Irr\big(\cH_\GG(\LL,\la)\big)$ into
Rouquier blocks. This stronger form fails, though, in general, even for ordinary
Harish-Chandra series, as the following examples show.

\begin{exmp}   \label{exmp:G25}
(1) Let $\GG$ be the spets with Weyl group $W=G_{25}$ and $\Phi=x-1$. The
 relative Weyl group of the $\Phi$-cuspidal pair $(\LL,\la)$ with $\LL$ a
 standard Levi subgroup of type $Z_3^2$ is cyclic of order~6 (see
 Table~\ref{tab1}). The relative Hecke algebra $\cH_\GG(\LL,\la)$ has five
 Rouquier blocks: four of the Schur elements do not lie in any proper ideal
 of $R_K(v)$ and thus their characters form singleton blocks
 \cite[Lem.~2.6(b)]{MR03}; this is related to the fact that $\ze_6-1=\ze_3$
 is a unit in $\ZZ_K$. On the other hand, the characters in
 $\Uch(\GG,(\LL,\la))$ are contained in three families of $\GG$.
 This cuspidal pair is also interesting in that $(\LL,\la)$ is invariant under
 all automorphisms of $\GG$ and also Galois invariant, but there does not
 exist a Galois invariant set of parameters for $\cH_\GG(\LL,\la)$.   \par
(2) While we did not find a counterexample among primitive involutive spetses,
 here is an imprimitive involutive one. Let $W=G(6,6,3)$, $\Phi=x-1$, $\LL$ the
 1-split standard Levi subgroup of type $G(6,6,2)$, that is, the Weyl group
 of type $G_2$, and $\la$ its 1-cuspidal character denoted $G_2[1]$. It is
 labelled by the symbol $(0,1; 0,1; 0; -; -; 0)$ (see \cite[Bsp.~6.6]{Ma95}).
 The relative Weyl group is again cyclic of order~6 and its cyclotomic Hecke
 algebra $\cH_\GG(\LL,\la)$ has parameters
 $\{-1,\ze_3,\ze_3^2x,-\ze_3^2x,x^2,-\ze_3x^2\}$ by \cite[Satz~6.10]{Ma95} and
 thus five Rouquier families, while the characters in $\Uch(\GG,(\LL,\la))$ lie
 in just three different families by \cite[6D]{Ma95}.
\end{exmp}

\section{Open problems}   \label{sec:prob}
We take the opportunity to list some open problems on unipotent characters
of spetses.

\begin{itemize}
\item It is conceivable that the partition of $\Uch(\GG)$ into families can be
characterised as being the finest partition such that the conclusion of
Theorem~\ref{thm:fam} holds simultaneously \emph{for all} $\Phi$-cuspidal pairs
for all $\Phi$.
\item Attached to families of unipotent characters there is a Fourier
matrix satisfying axioms similar to and generalising those of Lusztig's
Fourier transforms, as stated for example in \cite[Thm~6.9]{GM03} for Coxeter
groups or \cite[Folg~4.15]{Ma95} for imprimitive groups. Such matrices have
been constructed for the infinite series
$G(m,1,n)$ in \cite{Ma95} and have been found for all primitive groups by the
author (at present only published in {\tt Chevie} \cite{Mi15}). For the
infinite series of spetsial groups $G(m,m,n)$ such matrices are not known when
$\gcd(m,n)>1$, except for the cases $m=2$
(type $D_n$) and $n=2$ (dihedral groups) by the work of Lusztig \cite{Lu93}.

\item The Fourier matrices have been constructed in an ad hoc way in each
individual case (see the preprint \cite{LM26} for the description of a method
in the primitive cases). There is no general theory (as in the case of finite
reductive groups via the Drinfeld double of the group of components of the
centraliser of a suitable unipotent element) that would allow for an a priori
construction.

\item The parameters of relative Hecke algebras of $\Phi$-cuspidal pairs are
 constructed ad hoc. It would be desirable to find some systematic description
 for them.

\item There is no general formula for the behaviour of Frobenius eigenvalues
of unipotent characters under Ennola transforms (but see \cite[Ax.~4.13]{BMM14}
for the case of a cyclic relative Weyl group).
\end{itemize}

\section{Tables}   \label{sec:tab}
In this section we present the data related to the primitive spetses obtained
and used in the preceding proofs.

\subsection{Harish-Chandra series and parameters of relative Hecke algebras}
The first set of tables collects information on $\Phi$-Harish-Chandra series
and parameters of relative Hecke algebras as needed for the proof of
Theorem~\ref{thm:HC+HL}. In Table~\ref{tab1} for each primitive spets
$\GG=(V,W\vhi)$, with $W$ not a Coxeter group, for each $K$-cyclotomic
polynomial $\Phi$ dividing the order polynomial of $\GG$, and for each proper
$\Phi$-cuspidal pair $(\LL,\la)$ of $\GG$ we give the relative Weyl group
$W_{\LL,\la}:=W_\GG(\LL,\la)$ and
the admissible parameters $\bx$ for the corresponding relative Hecke algebra
$\cH_\GG(\LL,\la)=\cH(W_\GG(\LL,\la),\bx)$. Note that the Schur elements
of $\cH_\GG(\LL,\la)$ are not changed when multiplying all parameters at a
fixed reflection by the same non-zero scalar, since the Schur elements are
homogeneous rational functions of degree~0 in the parameters. We have also
included the exceptional twisted spets $^4G(3,3,3)$ which had not been treated
in the literature before; see Table~\ref{tab:4I3} for its unipotent degrees
and families
\smallskip

Throughout Table~\ref{tab1}, $\ze_n$ denotes a fixed primitive $n$th root of
unity, and $I:=\ze_4$. In Column~2 of Table~\ref{tab1} we specify the cyclotomic
polynomials by an entry $m$, or $m'$, meaning the cyclotomic polynomial
$\Phi\in K[X]$ with $\Phi(\ze_m)=0$, where for clarity we write $m'$ to
indicate that there is more than one Galois orbit of primitive $m$th roots of
unity over $K$. The parameters for other choices of primitive $m$th roots of
unity are obtained by applying a suitable Galois automorphism and are hence
omitted.

In order to simplify entries in the table, we note $Z_m:=G(m,1,1)$ for the
cyclic 1-dimensional reflection group of order~$m$, $B_2^{(m)}:=G(m,1,2)$,
and $I_3^{(3)}:=G(3,3,3)$, and $D_2^{(6)}:=G(6,2,2)$.

The $\Phi$-cuspidal characters $\la$ are identified via their name
in the tables of \cite{BMM14}.

In order to keep the size of the tables reasonable, we make
the following simplifications:
\begin{itemize}
 \item we only print one representative in each orbit under Ennola $d$-ality,
  where $d=|Z(W)|$ (using Theorem~\ref{thm:Ennola}(e));
 \item for involutive spetses we only print one representative in each orbit
  under Alvis--Curtis duality (using Theorem~\ref{thm:duality}(c));
 \item we only print one representative in each Galois orbit.
\end{itemize}

\begin{exmp}
(1) The $\Phi_1$-split Levi subgroup of $G_8$ of type $\Phi_1.Z_4$ has two
 $\Phi_1$-cuspidal characters, denoted $Z_4^{1022}$ and $Z_4^{1220}$ in
 \cite{BMM14}, that are Galois conjugate. In Table~\ref{tab1} we only print
 the information for the first of the two.   \par
(2) The $\Phi_4$-split Levi subgroup of $G_{24}$ of type $\Phi_4.A_1$ has two
 $\Phi_4$-cuspidal characters: the trivial one and the Steinberg. Since the two
 are related by Alvis--Curtis duality, we only print the information for the
 trivial character.   \par
(3) The $\Phi_9'$-split Levi subgroup of $G_{32}$ of type $\Phi_9'.Z_3$ is
 invariant under Ennola twisting by a 3rd root of unity, and this has an orbit
 of length~3 on its $\Phi_9'$-cuspidal characters. We only print the
 information for one representative of this orbit.
\end{exmp}

Most of the content of Table~\ref{tab1} is made available through {\tt Chevie}.

\begin{table}
\caption{Parameters for primitive spetses.}   \label{tab1}
$\begin{array}{cc|cc|c|l}
 \GG& d& \LL& \la& W_{\LL,\la}& \text{parameters}\cr
\hline\hline
 G_4& 1&  \Phi_1^2& 1& G_4& \ze_3, \ze_3^2, x\cr
          &  & \Phi_1.Z_3&  Z_3&       Z_2& 1, -x^3\cr
          &3'&\Phi_1\Phi_3'& 1&       Z_6& 1, -1, \ze_3, -x, -\ze_3x, x^2\cr
          & 4&    \Phi_4& 1&       Z_4& 1, -1, x, x^3\cr
\hline
 G_6& 1&  \Phi_1^2& 1& G_6& \ze_3, \ze_3^2, x; 1, -x\cr
          &  & \Phi_1.Z_3&  Z_3&       Z_4& 1, -x^3, Ix^3, -Ix^3\cr
          &3'&\Phi_1\Phi_3'& 1&    Z_{12}& 1, \ze_3^2, \pm Ix^\hlf,
     \pm x, -\ze_3^2x, \pm Ix, \pm\ze_{12}^5x, -x^2\cr
\hline
 G_8& 1&  \Phi_1^2& 1& G_8& 1, I, -I, -x\cr
          &  & \Phi_1.Z_4& Z_4^{1022}&       Z_4& 1, I, -Ix^2, -x^3\cr
          &  &          &  Z_4^{0212}&       Z_4& I, -I, -x^2, x^3\cr
          & 3&    \Phi_3& 1&    Z_{12}& \pm1, -\ze_3, -\ze_3^2, \pm I,
     \pm Ix^\hlf, x, \pm Ix, x^2\cr
          &8'&   \Phi_8'& 1&       Z_8& 1, -1, I, -I, x, -x, -Ix, x^3\cr
\hline
 G_{14}& 1&  \Phi_1^2& 1& G_{14}& 1, -x;\ \ze_3,\ze_3^2,x\cr
          &  & \Phi_1.Z_3&  Z_3& Z_6& 1,-x^3,\ze_3x^4,\ze_3^2x^4,-\ze_3x^4,
      -\ze_3^2x^4\cr
          & 4&    \Phi_4& 1&    Z_{24}& \ze_3,\ze_3^2,\pm x^\hlf,
      \pm Ix^\hlf,x^{2\over3},\ze_3x^{2\over3},\ze_3^2x^{2\over3},
      \pm x,\pm\ze_3x,\pm\ze_3^2x,\cr
          & & & & & \pm\ze_{24}x,\pm\ze_{24}^5x,\pm\ze_{24}^7x,
      \pm\ze_{24}^{11}x,-x^2\cr
          &8'&   \Phi_8'& 1&  Z_{24}& \ze_3,\ze_3^2,\pm\ze_{16}^5x^\hlf,
     \pm\ze_{16}^7x^\hlf,x^{2\over3},\ze_3x^{2\over3},\ze_3^2x^{2\over3},\pm x,\pm\ze_3x,\cr
          & & & & & \pm\ze_3^2x,\pm\ze_{12}x,\pm\ze_{12}^5x, \ze_{24}x, \ze_{24}^{11}x,
    \ze_{24}^{17}x, \ze_{24}^{19}x, x^2\cr
\hline
 G_{24}   & 1&  \Phi_1^3& 1&    G_{24}& 1,-x\cr
          &  & \Phi_1.B_2& B_2&       Z_2& 1, -x^7\cr
          & 3&\Phi_1\Phi_3& 1&       Z_6& 1, -x^3, x^{7\over2}, -x^{7\over2},
      x^4, -x^7\cr
          & 4& \Phi_4.A_1& 1&      Z_4& 1, -1, x, x^7\cr
          &7'&   \Phi_7'& 1&    Z_{14}& 1, \pm x, -\ze_7^3x, -\ze_7^5x,
      -\ze_7^6x, \pm x^{3\over2}, \pm x^2, \ze_7x^2, \ze_7^2x^2,\cr
          &  &          &  &          & \ze_7^4x^2, -x^3\cr
\hline
 G_{25}& 1&  \Phi_1^3& 1& G_{25}& \ze_3, \ze_3^2, x\cr
          &  &\Phi_1^2.Z_3& Z_3& B_2^{(3)}& \ze_3, \ze_3^2, x; 1, -x^3\cr
          &  & \Phi_1.G_4& G_4& Z_3& 1, \ze_3x^4, \ze_3^2x^4\cr
          &  & \Phi_1.Z_3^2& Z_3^2& Z_6& 1, -\ze_3, \ze_3^2x^2, -\ze_3^2x^2,
      \ze_3x^3, -x^3\cr
          & 2&\Phi_1\Phi_2^2& 1& G_5& \ze_3, \ze_3^2, -x;
      \ \ze_3, \ze_3^2, x^2\cr
          &  & \Phi_1\Phi_2.\tw2Z_3& \tw2Z_3& Z_6& \ze_3, \ze_3^2, x,
      \ze_3x^3, \ze_3^2x^3, x^4\cr
          &  & \Phi_2.Z_3(x^2)& Z_3& Z_6& \ze_3, \ze_3^2, x^2, -x^2,
      \ze_3x^3, \ze_3^2x^3\cr
          & 4&\Phi_1\Phi_4& 1&    Z_{12}& \pm1, \ze_3, \ze_3^2, x,
      \pm\ze_3x, \pm\ze_3^2x, \ze_3x^2, \ze_3^2x^2, x^3\cr
          &9'&   \Phi_9'& 1&       Z_9& 1, \ze_3,\ze_3^2, x, \ze_3^2x,
      x^2, \ze_3x^2, \ze_3^2x^2, x^4\cr
\hline
 G_{26}\!\!\!& 1&  \Phi_1^3& 1& G_{26}& 1, -x;\ \ze_3,\ze_3^2,x\cr
          &  & \Phi_1^2.Z_3& Z_3& D_2^{(6)}& 1, -x; 1, -x^3; 1, \ze_3x^2,
      \ze_3^2x^2\cr
          &  & \Phi_1.G_4& G_4&       Z_6& 1, -\ze_3x^3, -\ze_3^2x^3,
      -x^3, \ze_3x^4, \ze_3^2x^4\cr
          &  & \Phi_1.B_2^{(3)}&  &       Z_6& \ze_3^2, -x, \ze_3x,
     -\ze_3x, -\ze_3^2x^3, x^4\cr
          & 4& \Phi_4.A_1& 1&       Z_{12}& \ze_3, \ze_3^2, \pm x, \pm\ze_3x,
      \pm\ze_3^2x, x^2, \ze_3x^2, \ze_3^2x^2, x^4\cr
          &  &          &\St&       Z_{12}& \pm\ze_3^2, x, \ze_3x,
      \ze_3^2x, \pm x^2, \pm\ze_3x^2, x^3, \ze_3x^3, \ze_3^2x^3\cr
          &9'&   \Phi_9'& 1&       Z_{18}& \pm 1, -\ze_3, \pm x, \pm\ze_3x,
      \pm\ze_3^2x, \ze_9x, \ze_9^4x, \ze_9^7x, \pm x^{3\over2}, -x^2,\cr
          &  &          &  &          & 
      -\ze_3x^2, -\ze_3^2x^2, x^3\cr
\hline
\end{array}$
\end{table}

\begin{table}
$\begin{array}{cc|cc|c|l}
 \GG& d& \LL& \la& W_{\LL,\la}& \text{parameters}\cr
\hline\hline
 \tw4I_3^{(3)}\!\!\!& 1& \Phi_1\Phi_4& 1& Z_6& 1, \ze_3x, \ze_3^2x,
      -\ze_3x^2, -\ze_3^2x^2, -x^3\cr
              & 4& \Phi_4.A_1\!\!\!\!& 1& Z_3& \ze_3, \ze_3^2, x^4\cr
\hline
    G_{27}\!\!\!& 1&  \Phi_1^3& 1&    G_{27}& 1,-x\cr
          &  & \Phi_1.B_2& B_2& Z_6& \ze_3, \ze_3^2, -x, x^4, -\ze_3x^5,
      -\ze_3^2x^5\cr
          &  & \Phi_1.H_2& H_2& Z_6& \ze_3, \ze_3^2, x^{5\over2},
      -x^{5\over2}, -\ze_3x^5, -\ze_3^2x^5\cr
          & 4& \Phi_4.A_1&  1& Z_{12}& \pm\ze_3, \pm\ze_3^2, \ze_3x, \ze_3^2x,
      x^{5\over3}, \ze_3x^{5\over3}, \ze_3^2x^{5\over3}, \pm x^2, x^5\cr
          &5'& \Phi_1\Phi_5'& 1& Z_{30}& 1, \pm\ze_3x, \pm\ze_3^2x, 
      -\ze_{15}x, -\ze_{15}^4x, -\ze_{15}^{11}x,-\ze_{15}^{14}x,x^{4\over3},\cr
          & & & & & \ze_3x^{4\over3}, \ze_3^2x^{4\over3}, \pm x^{3\over2},
     \pm\ze_5^2x^{3\over2}, \pm\ze_5^3x^{3\over2}, -x^{5\over3},
     -\ze_3x^{5\over3}, \cr
          & & & & &   -\ze_3^2x^{5\over3}, \pm\ze_3x^2, \pm\ze_3^2x^2,
     \ze_{15}x^2, \ze_{15}^4x^2,\ze_{15}^{11}x^2, \ze_{15}^{14}x^2, -x^3\cr
\hline
    G_{29}\!\!\!& 1&  \Phi_1^4& 1&    G_{29}& 1,-x\cr
          &  &\Phi_1^2.B_2& B_2& B_2^{(4)}&1, I, -x^2, -Ix^2; 1, -x^3\cr
          &  &\Phi_1.I_3^{(4)}& I_3^{(4)}& Z_4& 1,Ix,-Ix,-x^6\cr
          & 3&\Phi_1\Phi_3.A_1& 1& Z_{12}& 1, \pm I, \pm x, -x^2, \pm Ix^2,
      \pm Ix^{5\over2}, -x^3, x^5\cr
          & 5&    \Phi_5& 1&    Z_{20}& 1, x, \pm Ix, \pm Ix^{3\over2},
      \pm x^2, \pm Ix^2, -\ze_5x^2,-\ze_5^2x^2, -\ze_5^3x^2,\cr
          &  &          &  &          & -\ze_5^4x^2, \pm Ix^{5\over2}, x^3,
      \pm Ix^3, x^4\cr
          & 8'& \Phi_8'.\tw4I_2^{(4)}& 1&  Z_8& 1, -1, x, -x, Ix, -Ix,
      x^2,x^6\cr
          &  &    &  &  Z_8& 1, -1, x^2, -x^2, x^3, -Ix^3, x^5, Ix^5\cr
\hline
 G_{32}\!\!\!\!& 1&  \Phi_1^4& 1& G_{32}& \ze_3,\ze_3^2,q\cr
          &  & \Phi_1^3.Z_3& Z_3& G_{26}& \ze_3, \ze_3^2, q;\ 1, -q^3\cr
          &  & \Phi_1^2.G_4& G_4& G_5& \ze_3, \ze_3^2, q;\  1,
      \ze_3q^4, \ze_3^2q^4\cr
          &  & \Phi_1^2.Z_3^2& Z_3^2& B_2^{(6)}& 1, -\ze_3,
      \ze_3^2q^2, -\ze_3^2q^2, \ze_3q^3, -q^3;\  1, -q^3\cr
          &  &  \Phi_1G_{25}&  G_{25}[\ze_3]&       Z_6& 1, -\ze_3q, -\ze_3^2q,
     \ze_3q^4, \ze_3^2q^4, -q^6\cr
          &  & & \!\!\!\!\!\!\!\!\!G_{25}[-\ze_3]\!\!\!\!\!&       Z_6& 1, \ze_3q, \ze_3^2q,
     -\ze_3q^4, -\ze_3^2q^4, -q^6\cr
          &  & \Phi_1.G_4Z_3& G_4Z_3& Z_6& 1, -\ze_3q^5, -\ze_3^2q^5,
     \ze_3q^8, \ze_3^2q^8, -q^9\cr
          & 4&  \Phi_4^2& 1& G_{10}& 1,-1,q,q^3;\ -\ze_3,-\ze_3^2,q^2\cr
          &  & \Phi_4.G_4& \phi_{2,1}&  Z_{12}& 1,\ze_3, \pm q,\pm\ze_3^2q,
     \ze_3q^2,\ze_3^2q^2,\pm\ze_3q^3,\ze_3^2q^4,q^6\cr
          &  &          & Z_3:2& Z_{12}& \ze_3, \ze_3^2, \pm\ze_3q,
      \pm\ze_3^2q, \ze_3q^2, \ze_3^2q^2, \pm q^3, q^4, q^6\cr
          &  &          & \phi_{1,4}& Z_{12}& 1,-\ze_3,\ze_3^2q,q^3,\ze_3q^3,
      \ze_3^2q^3,-q^4,\ze_3q^4,\pm\ze_3^2q^4, q^5,\ze_3q^5\cr
          &  &          & Z_3:1^2& Z_{12}& \pm\ze_3, \ze_3q, q^3, \ze_3q^3,
      \ze_3^2q^3, \pm q^4, \pm\ze_3^2q^4, q^5, \ze_3^2q^5\cr
          &  & \Phi_4.\tw2Z_3(q^2)\!\!\!& \tw2Z_3& Z_{12}& \ze_3,\ze_3^2,\pm q^2,\pm\ze_3q^3,\pm\ze_3^2q^3,\pm q^5,\ze_3q^6,\ze_3^2q^6\cr
          &9'& \Phi_9'.Z_3& 1&  Z_{18}& \pm1, \ze_3^2, \pm q, \pm\ze_3q,
      \pm\ze_3^2q, \pm q^2, \pm\ze_3q^2, \pm\ze_3^2q^2,-q^3,\cr
          &   &            &  &          & -\ze_3^2q^3, q^6\cr
          &   &            &  &    Z_{18}& \pm\ze_3, -\ze_3^2, q,
      \ze_3q, \pm q^2, \pm\ze_3^2q^2, -\ze_3q^2, \pm\ze_3q^{5\over2},q^3,\cr
          &   &            &  &          &  \pm\ze_3q^3, -\ze_3q^4,
      -\ze_3^2q^4, q^5\cr
          &24'& \Phi_{24}'& 1&    Z_{24}& \pm 1, \pm\ze_3^2, \pm q,
      \pm\ze_3q, \ze_3^2q, \ze_9^2q^{5\over3}, \ze_9^5q^{5\over3},
      \ze_9^8q^{5\over3}, \pm q^2, \cr
          &  &          &  &          & \pm\ze_3q^2, \pm\ze_3^2q^2, \pm Iq^2,
      q^3, \ze_3q^3, \ze_3^2q^3, q^5\cr
          &15'&  \Phi_{15}'& 1&    Z_{30}& \pm1, \pm \ze_3, \ze_3^2, \pm q, \pm \ze_3q,
      \pm \ze_3^2q, -\ze_5q, -\ze_5^2q, -\ze_5^3q,\cr
          &  &          &  &          &  -\ze_5^4q, \ze_9q^{4\over3},
      \ze_9^4q^{4\over3},\ze_9^7q^{4\over3}, \pm q^{3\over2},
      \pm\ze_3q^{3\over2},\pm q^2,\ze_3q^2,\cr
          &  &          &  &          & \pm\ze_3^2 q^2, -q^3,-\ze_3q^3,q^4\cr
\hline
\end{array}$
\end{table}

\begin{table}
$\begin{array}{cc|cc|c|l}
 \GG& d& \LL& \la& W_{\LL,\la}& \text{parameters}\cr
\hline\hline

    G_{33}\!\!\!&  1&  \Phi_1^5& 1&    G_{33}& 1,-q\cr
          &   & \Phi_1^2.I_3^{(3)}& & G_4& \ze_3, \ze_3^2, q^3\cr
          &   & \Phi_1.D_4& D_4&    Z_6& 1, -\ze_3q, -\ze_3^2q, \ze_3q^4,
     \ze_3^2q^4, -q^5\cr
          & 3'& {\Phi_3'}^3\Phi_1^2& 1& G_{26}& -\ze_3,q;\ \ze_3, \ze_3^2q,
      q^2\cr
          &   & \Phi_3'.\tw3D_4\!\!& \!\!\!\!\tw3D_4[-1]& Z_6& 1, -q, -\ze_3^2q, q^4, \ze_3q^4,
      -q^5\cr
          &  4& \Phi_4^2.A_1& 1& G_6& \ze_3, \ze_3^2,q^4;\ 1,q^2\cr
          &   & \Phi_4.\tw4I_3^{(3)}\!\!\!\!&  & Z_4& 1,-1,q^3,q^9\cr
          &   5& \Phi_5\Phi_1& 1&    Z_{10}& 1, q^3, -q^3, q^4, q^{9\over2},
      -q^{9\over2}, -q^5, q^6, -q^6, -q^9\cr
          & 9'& \Phi_1\Phi_3''\Phi_9'& 1& Z_{18}& 1, -q, -\ze_3q, \pm q^2,
      \pm\ze_3^2q^2, \ze_3q^2, \pm q^{5\over2}, \pm q^3, \pm\ze_3q^3, \cr
          &   &          &  &          & -\ze_3^2q^3, q^4, \ze_3^2q^4, -q^5\cr
          & 12'& \Phi_4\Phi_{12}'.A_1(\ze_3q)\!\!\!\!\!\!\!\!& 1&    Z_{12}& 1, \ze_3^2,
      \pm q, q^2, \ze_3^2q^2, \pm q^3, \pm\ze_3^2q^3, q^4, q^6\cr
\hline
    G_{34}\!\!&  1&  \Phi_1^6& 1&    G_{34}& 1,-q\cr
          &   &  \Phi_1^3.I_3^{(3)}& &  G_{26}& \ze_3, \ze_3^2, q^3; 1,-q\cr
          &   &  \Phi_1^2.D_4& D_4& B_2^{(6)}& 1, -\ze_3q, -\ze_3^2q,
      \ze_3q^4, \ze_3^2q^4, -q^5;\  1, -q^4\cr
          &   &  \Phi_1.G_{33}& \!\!\!\!G_{33}[-\ze_3^2]&    Z_6& \ze_3, \ze_3^2, -\ze_3q,
     -\ze_3^2q, -q^5, q^8\cr
          &   &                  & G_{33}[i]&    Z_6& 1, \ze_3, \ze_3^2q^2,
     -\ze_3^2q^5, -q^7, -\ze_3q^7\cr
          &  4&  \Phi_4^2.A_1^2& 1&    G_{10}& 1,q,-q,-q^4; \ze_3, \ze_3^2, q^4\cr
          &   &                  &  1\otimes\St&    G_7& \ze_3,\ze_3^2,q^4; 1, \ze_3q^4, \ze_3^2q^4; 1,q^2\cr
          &   & \Phi_4.\tw4I_3^{(3)}A_1& & Z_{12}& \ze_3, \ze_3^2, \pm q,
     \pm\ze_3q^3, \pm\ze_3^2q^3, q^4, \ze_3q^6, \ze_3^2q^6, q^{10}\cr
          &   &                  & & Z_{12}& \pm1, \ze_3q, \ze_3^2q, q^3,
     \pm\ze_3q^4, \pm\ze_3^2q^4, \ze_3q^7, \ze_3^2q^7, q^9\cr
          &  5&   \Phi_1\Phi_5.A_1& 1& Z_{30}& -\ze_3, -\ze_3^2, \pm q, -q^2,
      \pm\ze_3q^2, \pm\ze_3^2q^2, q^{7\over3},\ze_3q^{7\over3},
      \ze_3^2q^{7\over3}, \cr
          & & & & & \pm q^{5\over2}, \pm q^3, \pm\ze_3q^3, \pm\ze_3^2q^3,
      \pm\ze_3q^{7\over2},\pm\ze_3^2q^{7\over2}, \cr
          & & & & & -q^4, -\ze_3q^4, -\ze_3^2q^4,
     \ze_3q^5, \ze_3^2q^5, q^7\cr
          &  7& \Phi_7& 1& Z_{42}& 1, -\ze_3q, -\ze_3^2q, q^2, \pm\ze_3q^2,
     \pm\ze_3^2q^2, \pm\ze_3q^{5\over2}, \pm\ze_3^2q^{5\over2},\cr
          & & & & & q^{8\over3},
     \ze_3q^{8\over3}, \ze_3^2q^{8\over3}, \pm q^3, \pm\ze_3q^3,
     \pm\ze_3^2q^3, \cr
          & & & & & -\ze_7q^3, -\ze_7^2q^3, -\ze_7^3q^3,
     -\ze_7^4q^3, -\ze_7^5q^3, -\ze_7^6q^3,\cr
          & & & & & q^{10\over3}, \ze_3q^{10\over3}, \ze_3^2q^{10\over3},
     \pm\ze_3q^{7\over2}, \pm\ze_3^2q^{7\over2}, q^4, \pm\ze_3q^4,\cr
          & & & & & \pm\ze_3^2q^4, -\ze_3q^5, -\ze_3^2q^5, q^6\cr
          &  8& \Phi_8.A_1(q^2)& 1& Z_{24}& \ze_3, \ze_3^2, \pm\ze_3q,
     \pm\ze_3^2q, q^2, \ze_3q^2, \ze_3^2q^2, \pm q^3, \pm\ze_3q^3, \cr
          & & & & & \pm\ze_3^2q^3, q^{10\over3}, \ze_3q^{10\over3},
     \ze_3^2q^{10\over3}, q^4, \ze_3q^4, \ze_3^2q^4, \pm q^5, q^8\cr
          &  9'&  \Phi_9'.\tw3I_3^{(3)}& 1&    Z_{18}& 1, -q, -\ze_3q,
     -\ze_3^2q, \pm q^{3\over2}, \pm q^2, \pm\ze_3q^2,
     \pm\ze_3^2q^2,\pm q^3,\cr
          &   &  & &          &  q^4,\ze_3q^4, \ze_3^2q^4, -q^9\cr
          &   &  & &    Z_{18}& \pm\ze_3, -\ze_3^2, \pm q, \ze_3q,
     \pm\ze_3q^2, \pm\ze_3^2q^2, \ze_3q^3, \ze_3^2q^3, \pm\ze_3q^{7\over2},\cr
          &   &  & &          & 
      -q^4, -\ze_3q^4, -\ze_3q^6, q^7\cr
          &   &  & &    Z_{18}& \pm 1, -q, -\ze_3q, -\ze_3^2q, q^2,
     \ze_3q^2, \ze_3^2q^2, \pm q^3,q^4, \ze_3q^4, \cr
          &   &  & &          & \ze_3^2q^4, -q^5,
     -\ze_3q^5, -\ze_3^2q^5, \pm q^6\cr
\hline
\end{array}$
\end{table}

\newpage
\subsection{Families}
The second set of tables describes the families in $\Uch(\GG)$ for the
primitive spetses $\GG=(V,W\vhi)$ for $W$ not a Weyl group. In
Table~\ref{tab:fam} for each such $\GG$ we list all families
$\cF\subseteq\Uch(\GG)$ containing at least two elements. We give the
invariants $a=a(\cF)$ (the precise power of $x$ dividing $\deg(\rho)$ for
all $\rho\in\cF$), $A=A(\cF)$ (the degree in $x$ of $\deg(\rho)$), and the
size $|\cF|$ of $\cF$. We also display the distribution of the elements
of $\cF$ into the various $\Phi_1$-Harish-Chandra series of $\GG$. The
column headed $|E_\xi|$ gives the order of the action of Ennola $d$-ality
$E_\xi$ on the family (via the parameters of the associated Hecke algebras
as described in Theorem~\ref{thm:Ennola}(e)), with $\xi$ of order
$|Z(W)|^2$. Thus, the order of $E_\xi$ is a divisor of $|Z(W)|^2$.

To shorten the tables we make use of Alvis--Curtis duality, as described in
Theorem~\ref{thm:duality}: for the
spetsial groups generated by reflections of order 2, viz.~the groups
$$G_{23},G_{24},G_{27},G_{29},G_{30},G_{33}\text{ and }G_{34},$$
as well as for $\tw4I_3^{(3)}$, we only print one representative in each orbit
of $D_\GG$ on the set of families. Families fixed under Alvis--Curtis duality
are indicated with a $*$. Note that by Theorem~\ref{thm:duality}, for any
family $\cF$, if $\cF'=D_\GG(\cF)$, then
$$a(\cF)+A(\cF')=a(\cF')+A(\cF)=N_W.$$

\begin{table}[h]
\caption{Families for primitive spetses}   \label{tab:fam}
$\begin{array}{crr|rr|rrrrc}
  \GG& a& A& |\cF|& |E_\xi|& 1& Z_3& \text{cusp.} \cr
  \hline\hline
G_4& 1&  5&  3&   2&  2& 1&  -& &\cr
 &   4&  8&  5&   2&  3& 1&  1& &\cr
  \hline  
   & & & & & 1& Z_3& \text{cusp.} \cr
  \hline
G_6& 1& 11& 22&   4& 8& 3& 11& & \cr
 &   5& 13&  4&   4& 2& -&  2& &\cr
 &  10& 14&  3&   1& 2& 1&  -& &\cr
  \hline  
   & & & & & 1& Z_4& \text{cusp.} \cr
  \hline
G_8& 1& 11&  6&   4& 3& 3&  -& &\cr
 &   2& 14&  6&   2& 3& 3&  -& &\cr
 &   3& 15&  4&   4& 2& -&  2& &\cr
 &   6& 18& 18&   4& 7& 6&  5& \cr
  \hline
  & & & & & 1& Z_3& \text{cusp.} \cr
  \hline
G_{14}& 1& 23& 54&   6& 13& 4& 37& \cr
 &   5& 25&  3&   3&  2& 1&  -& &\cr
 &   6& 26&  9&   9&  3& -&  6& &\cr
 &   9& 27&  9&   4&  3& -&  6& & \cr
 &  20& 28&  3&   3&  2& 1&  -& &\cr
  \hline
  & & & & & 1& \text{cusp.}\cr
  \hline
\tw4I_3^{(3)}& 1& 5& 3& 3& 2& 1& & \cr
  \hline
   & & & & & 1& H_2& \text{cusp.}\cr
  \hline
G_{23}&    1&  9&  4&    2& 2& 2&  -& \cr
 &   3& 12&  4&        4& 2& -&  2& & *\cr
  \hline  
   & & & & & 1& B_2& \text{cusp.}\cr
  \hline
G_{24}&    1& 13&  7&    2& 3& 1&  3& \cr
 &   4& 17&  4&    4& 2& -&  2& & *\cr
  \hline
\end{array}$
\end{table}

\begin{table}[htb]
$\begin{array}{crr|rr|rrrrrrrc}
  \GG& a& A& |\cF|& |E_\xi|& 1& Z_3& G_4& Z_3^2& \text{cusp.}\cr
  \hline\hline
G_{25}&    1& 11&  3&      3& 2& 1& -& -& -& \cr
  &  2& 16&  9&      3& 5& 2& -& 2& -& \cr
  &  4& 20& 15&      3& 7& 2& 2& 2& 2& \cr
  &  6& 21&  3&      1& 2& 1& -& -& -& \cr
  &  8& 22&  3&      3& 2& 1& -& -& -& \cr
  & 12& 24& 10&      3& 5& 2& 1& 2& -& \cr
  \hline  
  & & & & & 1& Z_3& G_4& B_2^{(3)}& \text{cusp.}\cr
  \hline
G_{26}&  1& 17&  9&  3& 5& 2& -& 2& -& \cr
  &  2& 22&  9&      3& 5& 2& -& 2& -& \cr
  &  3& 24&  4&      4& 2& -& -& -& 2& \cr
  &  4& 26& 12&      6& 6& 2& 2& -& 2& \cr
  &  5& 25&  3&      3& 2& 1& -& -& -& \cr
  &  6& 30& 49&      6& 17& 7& 3& 6& 16& \cr
  & 11& 31&  9&      3& 5& 2& -& 2& -& \cr
  & 16& 32&  3&      3& 2& 1& -& -& -& \cr
  & 21& 33&  5&      2& 3& 1& 1& -& -& \cr
  \hline
  & & & & & 1& B_2& H_2& \text{cusp.}\cr
  \hline
G_{27}&  1& 29& 18&  6& 6& 2& 4& 6&  \cr
  &  3& 33&  4&      2& 3& 1& -& -& \cr
  &  4& 36&  9&      9& 3& -& -& 6& \cr
  &  5& 37&  3&      3& 2& -& -& 1& \cr
  &  6& 39& 16&      4& 4& -& 4& 8&  & *\cr
  \hline
  & & & & & 1& B_2& I_3^{(4)}& \text{cusp.}\cr
  \hline
G_{29}& 1& 19&  6&    4& 3& 2& 1& -& & \cr
  &  3& 27&  4&       4& 2& -& -& 2&  & \cr
  &  4& 28&  4&       2& 3& 1& -& -&  & \cr
  &  5& 31&  6&       4& 3& 2& 1& -&  & \cr
  &  6& 34& 22&       4& 9& 4& 4& 5&  & *\cr
  \hline
  & & & & & 1& H_2& H_3& \text{cusp.}\cr
  \hline
G_{30}& 1& 29& 4&     2& 2& 2& -& -&  & \cr
  &  2& 38&  4&       2& 2& 2& -& -&  & \cr
  &  3& 42&  4&       4& 2& -& 2& -&  & \cr
  &  6& 54& 74&       2& 16& 8& -&50&  & * \cr
  \hline
\end{array}$
\end{table}

\begin{table}[htb]
$\begin{array}{crr|rr|rrrrrccc}
  \GG& a& A& |\cF|& |E_\xi|& 1& Z_3& G_4& Z_3^2&
     G_{25}& G_4\times Z_3& \text{cusp.}\cr
  \hline\hline
G_{32}&  1& 29& 3&   3& 2& 1& -& -& -& -& -&  \cr
  &  2& 46&  9&      3& 5& 2& -& 2& -& -& -&  \cr
  &  3& 51&  3&      1& 2& 1& -& -& -& -& -&  \cr
  &  4& 56& 15&      6& 7& 2& 2& 2& 2& -& -&  \cr
  &  5& 61& 15&      6& 6& 5& 2& 1& -& 1& -&  \cr
  &  6& 66& 42&      6& 16& 6& 5& 4& 4& 2& 5& \cr
  &  8& 67& 12&     12& 4& 2& -& -& -& -& 6&  \cr
  &  9& 69&  9&      3& 4& 4& -& 1& -& -& -&  \cr
  & 10& 70&  9&      9& 3& -& -& -& -& -& 6&  \cr
  & 12& 72& 45&      6& 17& 9& 3& 8& -& -& 8& \cr
  & 15& 75& 40&      6& 14& 6& 4& 4& 4& 2& 6& \cr
  & 20& 76& 15&      3& 7& 2& 2& 2& 2& -& -&  \cr
  & 25& 77&  9&      3& 4& 4& -& 1& -& -& -&  \cr
  & 30& 78&  5&      2& 3& 1& 1& -& -& -& -&  \cr
  & 40& 80& 15&      6& 7& 3& 2& 2& -& 1& -&  \cr
  \hline
  & & & & & 1& I_3^{(3)}& D_4& \text{cusp.}\cr
\hline
G_{33}& 1& 17& 3&    1& 2& 1& -& -&  & \cr
 & 3& 27&  4&        2& 3& -& 1& -&  & \cr
 & 4& 32& 15&        2& 7& 5& 2& 1&  & \cr
 & 7& 35&  3&        1& 2& 1& -& -&  & \cr
 & 8& 37&  4&        4& 2& -& -& 2&  & & *\cr
\hline
  & & & & & 1& I_3^{(3)}& D_4& G_{33}& \text{cusp.}\cr
\hline
G_{34}& 1& 41& 3&    3& 2& 1& -& -& -&  \cr
 & 2&  58&  3&       3& 2& 1& -& -& -&  \cr
 & 3&  69&  4&       2& 3& -& 1& -& -&  \cr
 & 4&  80& 15&       6& 7& 5& 2& 1& -&  \cr
 & 5&  85&  9&       3& 5& 4& -& -& -&  \cr
 & 6&  90&  9&       3& 5& 4& -& -& -&  \cr
 & 7&  95& 15&       6& 7& 5& 2& 1& -&  \cr
 & 8&  97& 12&      12& 4& 2& -& 4& 2&  \cr
 & 9&  99&  4&       2& 3& -& 1& -& -&  \cr
 &10&  98&  9&       3& 5& 4& -& -& -&  \cr
 &10& 102&  9&       9& 3& -& -& -& 6&  \cr
 &11& 103&  3&       3& 2& 1& -& -& -&  \cr
 &11& 103&  9&       3& 5& 4& -& -& -&  \cr
 &12& 105&  4&       4& 2& -& -& 2& -&  \cr
 &13& 107& 36&       6& 15& 10& 5& 2& 4& \cr
 &14& 106&  3&       3& 2& 1& -& -& -&  \cr
 &15& 111& 44&       6& 16& 12& 4& 4& 8& & * \cr
 &18& 108&  4&       2& 3& -& 1& -& -& & * \cr
\hline  
\end{array}$
\end{table}

\newpage{\ }
\newpage{\ }
\newpage{\ }


\end{document}